\documentclass{amsart}
\usepackage[utf8]{inputenc}
\usepackage{a4wide}
\usepackage{amsmath,amsfonts,amssymb}
\usepackage{amsthm}
\usepackage{mathtools}
\usepackage{ifthen}
\usepackage{longtable}
\usepackage{array}
\usepackage{url}
\usepackage{enumerate}
\usepackage{xcolor}
\usepackage{float}
\usepackage{subcaption}
\usepackage{graphicx}
\usepackage{todonotes}
\usepackage{hyperref}
\usepackage{cleveref}

\usepackage{titlesec}
\titleformat{\section}
  {\center\normalsize\bfseries}{\thesection.}{1em}{}
\titleformat{\subsection}
  {\normalsize\bfseries}{\thesubsection.}{1ex}{}

\hypersetup{colorlinks,
	linkcolor=green!50!black,
	citecolor=blue!70!black,
	urlcolor=red!50!black
}

\newcommand{\Z}{\mathbb{Z}}

\renewcommand{\P}{\mathbb{P}}

\newcommand{\eps}{\varepsilon}

\newcommand{\odd}{\text{odd}}
\newcommand{\even}{\text{even}}

\DeclarePairedDelimiter\ceil{\lceil}{\rceil}
\DeclarePairedDelimiter\floor{\lfloor}{\rfloor}
\DeclarePairedDelimiter\abs{\lvert}{\rvert}

\theoremstyle{definition}

	\newtheorem{problem}{Problem}
	\newtheorem{theorem}{Theorem}
    \newtheorem{lemma}{Lemma}
    \newtheorem{proposition}{Proposition}
    \newtheorem{remark}{Remark}

\newcommand{\seqnum}[1]{\href{https://oeis.org/#1}{\rm \underline{#1}}}

\begin{document}
\title[Sets of remainders]{Bounds for sets of remainders}
\subjclass[2020]{11N37, 11A07, 11B83} 
\keywords{Number of remainders, integer sequence, asymptotics, gaps, iterated remainders}
\thanks{Research of the second author supported by NSERC grant 2024-03725.
}

\author[O. Baraskar]{Omkar Baraskar}
\address{O. Baraskar,
School of Computer Science
University of Waterloo
Waterloo, ON N2L 3G1
Canada}
\email{obaraska\char'100uwaterloo.ca}

\author[I. Vukusic]{Ingrid Vukusic}
\address{I. Vukusic,
School of Computer Science
University of Waterloo
Waterloo, ON N2L 3G1
Canada}
\email{ingrid.vukusic\char'100uwaterloo.ca}

\begin{abstract}
Let $s(n)$ be the number of different remainders $n \bmod k$, where $1 \leq k \leq \floor{n/2}$. This rather natural sequence is sequence A283190 in the OEIS and while some basic facts are known, it seems that surprisingly it has barely been studied.
First, we prove that $s(n) = c \cdot n + O(n/(\log n \log \log n))$, where $c$ is an explicit constant. 
Then we focus on differences between consecutive terms $s(n)$ and $s(n+1)$.
It turns out that the value can always increase by at most one, but there exist arbitrarily large decreases. We show that the differences are bounded by $O(\log \log n)$.
Finally, we consider ``iterated remainder sets''. These are related to a problem arising from Pierce expansions, and we prove bounds for the size of these sets as well.
\end{abstract}

\maketitle

\section{Introduction}

Let us fix an integer $n$. Then a natural question is the following: Which are the remainders $r = n \bmod k$?
The sum of such remainders with $1 \leq k \leq n$ was already considered by \'Edouard Lucas \cite[p.~373]{Lucas1891} and more recently in \cite{E2817, Spivey2005, HoseanaAziz2021}.
In the present paper, we are simply interested in the number of distinct remainders.
Surprisingly, it seems that this has barely been studied.

If $k > n$, then we get $r = n$ for all $k$.

If $\floor{n/2} + 1 \leq k \leq n$, then $r = n-k$, so in this range we simply get all the integers between $0$ and $\ceil{n/2}-1$.

Therefore, the interesting cases are where $1 \leq k \leq \floor{n/2}$, and we define the set
\[
    S(n) 
    := \{ n \bmod k \colon k \in \{ 1,2, \ldots, \floor{n/2} \}
    \}.
\]
Note that clearly $S(n) \subseteq \{0,1, \ldots , \floor{n/2}-1\}$. 
Now let us define
\[
    s(n) := \abs{S(n)}.
\]
In other words, $s(n)$ is the number of different values $n \bmod k$ for $1 \leq k \leq \floor{n/2}$.
This is precisely sequence \seqnum{A283190} in the OEIS (On-Line Encyclopedia of Integer Sequences) \cite{oeis}.
The first few numbers of the sequence are presented in Table~\ref{tab:s}.

\begin{table}[h]
\centering
\begin{tabular}{c|cccccccccccccccccc}
$n$      & 1 & 2 & 3 & 4 & 5 & 6 & 7 & 8 & 9 & 10 & 11 & 12 & 13 & 14 & 15 & 16 & 17 & 18\\
\hline
$s(n)$    & 0 & 1 & 1 & 1 & 2 & 1 & 2 & 2 & 2 & 3  & 4  & 2  & 3  & 3  & 3  & 4  & 5  & 3  \\
\end{tabular}
\caption{First eighteen values of $s(n)$.}
\label{tab:s}
\end{table}

Although it seems like a very natural sequence, it was entered into the OEIS only in 2017 by Thomas Kerscher. 
Robert Israel then observed that $s(n)/n$ seems to converge to approximately $0.2296$. He asked about the actual value of the constant on the StackExchange website \cite{Israel2017}.
This was answered by the user Empy2 (and added as a comment to the OEIS entry by Michael Peake):
\begin{equation}\label{eq:asymp}
    \lim_{n \to \infty} \frac{s(n)}{n}
    = \sum_{p \in \P} \frac{1}{p(p+1)} 
        \cdot \prod_{\substack{q \in \P, \\ q < p}} \left( 1 - \frac{1}{q} \right)
    \approx  0.2296. 
\end{equation}
    
The answer on the StackExchange website also contains an explanation for why this is true.
While the argument is relatively simple, we haven't been able to find a rigorous proof in print. 
In this paper, in Section~\ref{sec:asymp}, we give a proof of \eqref{eq:asymp}, including a bound for the error term.

Then we investigate the differences between consecutive terms of $s(n)$.
Looking at the first few values of the sequence $s(n)$, one sees that the values of $s(n+1)$ compared to $s(n)$ usually either stay the same or increase or decrease by $1$. They never seem to increase by more than $1$, but sometimes they decrease by $2$ or, as it turns out, even more.
In Section~\ref{sec:diffs} we show that these differences can get arbitrarily large, while being double logarithmically bounded in terms of $n$.

In Section~\ref{sec:Pierce}, we generalize the set $S(n) =: S_1(n)$ to ``iterated remainder sets'' via $S_{j+1}(n) := \{n \bmod{k} \colon k \in S_{j}(n)\setminus \{0\}\}$ for $j \geq 1$.
These sets are related to an older problem arising from Pierce expansions. This problem was considered in \cite{Shallit1986, ErdosShallit1991, ChasePandey2022}, and the main question is still open: If we fix a positive integer $n$, choose another integer $1 \leq a \leq n$, and repeatedly set $a : = n \bmod a$, what is the largest number of steps performed before reaching $a = 0$? More on this and the relation to our iterated remainder sets in Section~\ref{sec:Pierce}.
In Section~\ref{sec:iterated_results} we prove some bounds for these sets.

Finally, in Section~\ref{sec:problems} we pose some open problems.

But first of all, we present all main results in the next section.

\section{Main results and some lemmas}

Let us define the constant from \eqref{eq:asymp}, namely
\[
    c := \sum_p \frac{1}{p(p+1)}
        \cdot \prod_{\substack{p' < p}} \left( 1 - \frac{1}{p'} \right)
        \approx 0.2296.
\]
In Section~\ref{sec:asymp} we will prove the following asymptotics for $s(n)$.

\begin{theorem}\label{thm:asymp}
We have
\[
    s(n) = c \cdot n + O \left( \frac{n}{\log n \log \log n} \right).
\]
In particular,
\[
    \lim_{n \to \infty} \frac{s(n)}{n} = c.
\]
\end{theorem}
We do not believe that our bound for the error term is sharp. Numerical experiments suggest that perhaps $O(n^{1/3})$ might be closer to the truth; see Section~\ref{sec:problems}.
In any case, we know that overall $s(n)$ grows linearly, but due to the error term this does not give us much information on the differences $s(n+1) - s(n)$.

As mentioned in the introduction, looking at the first few values of $s(n)$, it seems that the values of $s(n+1)$ compared to $s(n)$ usually either stay the same or increase or decrease by $1$. It turns out that they never increase by more than $1$, but sometimes they do decrease by $2$ or even more.
For example, $s(131) = 33$ and $s(132) = 30$. The first time that the value decreases by $4$ happens at $n = 17291$, where $s(17291) = 3975$ and $s(17292) = 3971$. We have searched up to $n = 10^7$ and have not found a decrease by more than $4$ in this range.
However, it turns out that there do exist arbitrarily large decreases, and we will give a construction for such $n$. On the other hand, we can prove bounds on the decreases in terms of $n$. In Section~\ref{sec:diffs} we will, in particular, prove the following results.

\begin{theorem}\label{thm:diffs}
For $n\geq 1$ we have
\[
    s(n) - O(\log \log n)
    \leq s(n+1) 
    \leq s(n) + 1.
\]
Moreover, 
\[
    \liminf_{n\to \infty} s(n+1) - s(n) = - \infty.
\]
\end{theorem}

Now let us define the sets of iterated remainders of $n$ inductively by
\begin{align*}
    S_0(n) &: = \{1,2, \ldots, \floor{n/2}\}
    \quad \text{and} \\
    S_{j}(n) &:= \{n \bmod{k} \colon k \in S_{j-1}(n)\setminus \{0\}\}
    \quad \text{for } j \geq 1.
\end{align*}
Note that, in particular, $S(n) = S_1(n)$. Moreover, analogously to $s(n)$, we define 
\[
    s_j(n) := \abs{S_j(n)}.
\]

In Section~\ref{sec:iterated_results}, we will prove the following bounds.

\begin{theorem}\label{thm:iterated_bounds}
For $j \geq 0$ we have
\[
    \frac{1}{(j+2)!}
    \leq \liminf_{n \to \infty} \frac{s_j(n)}{n}
    \leq \limsup_{n \to \infty} \frac{s_j(n)}{n}
    \leq \frac{1}{j+2}.
\]
\end{theorem}
While the bounds are clearly not sharp (for example, set $j=1$ and compare to Theorem~\ref{thm:asymp}), it seems that for $j\geq 2$ there is indeed a gap between the limit inferior and limit superior; see Section~\ref{sec:problems}.

Before moving on, we state two simple lemmas.
All results on $s(n)$ will be based on the following equivalence.

\begin{lemma}\label{lem:member-equiv}
Let $r \in \{0,1, \ldots , \floor{n/2}-1\}$. 
Then $r \in S(n)$ if and only if $n-r$ has a proper divisor $k \geq r+1$.
\end{lemma}
\begin{proof}
First, assume $r \in S(n)$. This means that there exist integers $k \in \{1,2, \ldots, \floor{n/2}\}$ and $q$ such that 
\[
    n = k\cdot q + r
\]
and $r \leq k-1$.
In other words, $k$ is a divisor of $n-r$ with $k \geq r+1$.
Moreover, $k\leq \floor{n/2}$ implies $q\geq 2$, and therefore $k$ is indeed a proper divisor of $n-r$.

Now assume conversely that $n-r$ has a proper divisor $k \geq r+1$. Then we have
\[
    n - r= k \cdot q
\]
for some $q \geq 2$.
This means that $r = n \bmod k$.
Moreover, $q \geq 2$ implies $k = (n-r)/q \leq \floor{n/2}$,
and so by definition $r \in S(n)$.
\end{proof}

Finally, let us give a slightly more precise statement than $S(n) \subseteq \{0,1, \ldots \floor{n/2} - 1\}$.

\begin{lemma}\label{lem:upper_bound_nover3}
We have $S(n) \subseteq \{0,1, \ldots, \floor{(n-2)/3}\}$.
\end{lemma}
\begin{proof}
Let $r \in S(n)$. Then by Lemma~\ref{lem:member-equiv}, the number $n-r$ has a proper divisor $k \geq r+1$. In other words, $n - r = k \cdot q$ with $q \geq 2$. This implies $n - r \geq (r+1) \cdot 2$, and rewriting the inequality yields $n \geq 3r+2$.
\end{proof}

Note that throughout the paper, $p$ will denote a prime. In particular, if we sum or take the union over an index $p$, the numbers $p$ are implied to be primes.

\section{Asymptotics for $s(n)$ (proof of Theorem~\ref{thm:asymp})}\label{sec:asymp}

In this section we want to prove that $s(n)$ asymptotically behaves like $c \cdot n$ and compute a bound for the error term. We will do this with a straightforward sieving argument.
In preparation for this, let us define the set of integers that are divisible by $p$ but not by any smaller prime $p'<p$: 
\begin{equation}\label{eq:Dp}
    D_p := \{m \in \Z \colon 
        p \mid m, \text{ and }
        p' \nmid m \text{ for every prime } p' < p
        \}.
\end{equation}
As usual, $\pi(n)$ will denote the number of primes $p \leq n$. 

The next lemma is standard; we provide a proof for completeness.

\begin{lemma}\label{lem:Dp}
Let $a, t$ be positive integers and $p$ a prime.
Then 
\[
    \abs{D_p \cap [a+1,a+t]} 
    = t \cdot \frac{1}{p} 
        \cdot \prod_{\substack{p' < p}} \left( 1 - \frac{1}{p'} \right)
        + E(a,t,p),
\]
and the error term $E(a, t, p)$ is bounded by
\[
    \abs{E(a,t,p)}
    \leq 2^{\pi(p-1)}.
\]
\end{lemma}
\begin{proof}
This follows from a simple inclusion--exclusion argument. Indeed, we have
\begin{align*}
    \abs{D_p \cap [a+1,a+t]} = 
    &\left( \floor*{\frac{a + t}{p}} - \floor*{\frac{a}{p}} \right) 
        - \sum_{p_1 < p} \left( \floor*{\frac{a + t}{p p_1}} - \floor*{\frac{a}{p p_1}} \right) \\
    & \quad    + \sum_{p_1<p_2<p} \left( \floor*{\frac{a + t}{p p_1 p_2}} - \floor*{\frac{a}{p p_1 p_2}} \right)  
        -+ \dots \\
    & \quad \pm \left( \floor*{\frac{a + t}{p \prod_{p'<p} p'}} - \floor*{\frac{a}{p \prod_{p'<p} p'}} \right).       
\end{align*}
Each difference of the shape $\floor{(a+t)/q} - \floor{a/q}$ is equal to $t/q + \delta$, with some error term $\delta$ with $\abs{\delta}< 1$. There are exactly $2^{\pi(p-1)}$ such differences. 
Therefore,
\begin{align*}
    \abs{D_p \cap [a+1,a+t]} = 
    & \frac{t}{p}
        - \sum_{p_1 < p} \frac{t}{p p_1}
        + \sum_{p_1<p_2<p} \frac{t}{p p_1 p_2}  
        -+ \dots
        \pm \frac{t}{p \prod_{p'<p} p'}
        + E(a,t,p)\\
    & = t \cdot \frac{1}{p} 
        \cdot \prod_{\substack{p' < p}} \left( 1 - \frac{1}{p'} \right)
        + E(a,t,p),
\end{align*}
with $\abs{E(a,t,p)} < 2^{\pi(p-1)}$, as desired.
\end{proof}

We will also use the next lemma, which is a classical upper bound for the prime counting function \cite[Corollary 1]{RosserSchoenfeld1962}.

\begin{lemma}\label{lem:primecounting_ub}
For every $n\geq 2$ we have
\[
    \pi(n) \leq 1.3 \cdot \frac{n}{\log n}.
\]
\end{lemma}

Finally, we will use the following simple estimate.

\begin{lemma}\label{lem:psquared_sum_est}
We have
\[
    \sum_{p > n} \frac{1}{p^2} 
    = O\left( \frac{1}{n \log n} \right).
\]
\end{lemma}
\begin{proof}
By the prime number theorem, the $k$-th $p_k$ prime is asymptotically of the size $p_k \sim k \log k$.
The lemma follows from a density argument. To be rigorous, we can do the following estimates for sufficiently large $n$, where $C_1,C_2,C_3$ are some positive constants:
\begin{align*}
    \sum_{p > n} \frac{1}{p^2} 
    &\leq C_1 \cdot \sum_{k \geq n/\log n} \frac{1}{(k \log k)^2}
    \leq C_2 \cdot \int_{n/\log n}^\infty \frac{1}{(t \log t)^2} \, dt\\
    &\leq C_2 \cdot \frac{1}{(\log(n/\log n))^2} \cdot \int_{n/\log n}^\infty \frac{1}{t^2} \, dt\\
    &\leq C_3 \cdot \frac{1}{(\log n)^2} \cdot \frac{1}{n /\log n}
    = C_3 \cdot \frac{1}{n \log n}.
\end{align*}
\end{proof}

Now we are ready to prove the asymptotics for $s(n)$.

\begin{proof}[Proof of Theorem~\ref{thm:asymp}]
We want to use Lemma~\ref{lem:member-equiv} to count the remainders $r \in S(n)\setminus \{0\}$.
Note that $n-r$ having a proper divisor $k \geq r+1$ is equivalent to $n-r$ having a prime factor $p \leq (n-r)/(r+1)$. 
Moreover, note that the last inequality is equivalent to $r \leq (n-p)/(p+1)$.
We want to count the remainders $r$ systematically by going through the prime numbers $p$.
Therefore, we define for every $n$ and every prime $p < n$ the set
\[
    R_p(n) := \left\{ r \colon 
        1 \leq r \leq \frac{n-p}{p+1}, \,
        p \mid n-r, \text{ and }
        p' \nmid n-r \text{ for every prime } p' < p
        \right\}.
\]
Then we have
\[
    S(n)\setminus \{0\}
    = \bigcup_{p < n} R_p(n).
\]
Since all sets in the union are disjoint by construction, we have in particular
\begin{equation}\label{eq:S-R}
    s(n) - 1
    =\abs{S(n)\setminus \{0\}} 
    = \sum_{p < n} \abs{R_p(n)}.
\end{equation}
Instead of counting the numbers $r \in R_p(n)$, we can equivalently count the numbers $m = n-r$. In other words, if we define
\[
    M_p(n) := \left\{ m \in \left[n -  \frac{n-p}{p+1}, n-1\right] \colon
        p \mid m, \text{ and }
        p' \nmid m \text{ for every prime } p' < p
        \right\},
\]
then we have $\abs{R_p(n)} = \abs{M_p(n)}$. 
Let $X = X(n)$ be some threshold function with $1 < X(n) < n$ that we will fix later.
Our strategy is to compute
\begin{equation}\label{eq:sn_1}
    s(n) 
    = 1 + \sum_{p < n} \abs{M_p(n)}
    = \sum_{p < X} \abs{M_p(n)} + \underbrace{1 + \sum_{X \leq p < n} \abs{M_p(n)}}_{=: E_1(n)}.
\end{equation}
We do this because on the one hand, if $p$ is small compared to $n$, we can compute $\abs{M_p(n)}$ relatively precisely. On the other hand, if $p$ is sufficiently large, $\abs{M_p(n)}$ is small, and we can estimate the error term $E_1(n)$ trivially. Let us do the latter first.
We have
\begin{equation}\label{eq:E1}
    \abs{E_1(n)} 
    = 1 + \sum_{X \leq p < n} \abs{M_p(n)}
    \leq 1 + \sum_{p \geq X} \frac{n}{(p+1)p}
    \leq 1 + n \cdot \sum_{p \geq X} \frac{1}{p^2}
    = O \left( \frac{n}{X \log X} \right),
\end{equation}
where for the last estimate we used Lemma~\ref{lem:psquared_sum_est}.

In order to compute the main term in \eqref{eq:sn_1}, for each $p < X$ we apply Lemma~\ref{lem:Dp} 
with $a = n - \frac{n-p}{p+1} - 1$ and $t = \frac{n-p}{p+1}$.
This gives us
\[
    \abs{M_p(n)}
    = \frac{n-p}{p+1} \cdot \frac{1}{p} 
        \cdot \prod_{\substack{p' < p}} \left( 1 - \frac{1}{p'} \right)
        + E_2(n,p),
\]
with 
\[
    \abs{E_2(n,p)} \leq 2^{\pi(p-1)}.
\]
Now we can write
\begin{align*}
    \sum_{p < X} \abs{M_p(n)}
    &= \sum_{p<X} \frac{n-p}{p+1} \cdot \frac{1}{p} 
        \cdot \prod_{\substack{p' < p}} \left( 1 - \frac{1}{p'} \right)
        + \sum_{p<X} E_2(n,p)\\
    &= n \cdot \sum_{p<X} \frac{1}{p(p+1)} 
        \cdot \prod_{\substack{p' < p}} \left( 1 - \frac{1}{p'} \right)
        \underbrace{ - \sum_{p<X} \frac{1}{p+1} 
            \cdot \prod_{\substack{p' < p}} \left( 1 - \frac{1}{p'} \right)
            + \sum_{p<X} E_2(n,p)
        }_{=: E_2(n)},
\end{align*}
with
\begin{equation}\label{eq:E2}
    \abs{E_2(n)} 
    \leq X + \sum_{p < X} 2^{\pi(p-1)}
    \leq X + 2^{\pi(X)},
\end{equation}
where in the error term we estimated the first sum very crudely by $X$, and for the second sum we used the definition of $\pi(x)$.

Finally, recall that
\[
    c = \sum_p \frac{1}{p(p+1)}
        \cdot \prod_{\substack{p' < p}} \left( 1 - \frac{1}{p'} \right).
\]
Thus, we can write 
\begin{equation}\label{eq:sn_2}
    \sum_{p < X} \abs{M_p(n)}
    = c \cdot n 
        - \underbrace{n \cdot \sum_{p \geq X}  \frac{1}{p(p+1)}
        \cdot \prod_{\substack{p' < p}} \left( 1 - \frac{1}{p'} \right)}_{=:- E_3(n)}
        + E_2(n),
\end{equation}
with
\begin{equation}\label{eq:E3}
    \abs{E_3(n)} 
    \leq n \cdot \sum_{p \geq X} \frac{1}{p^2}
    = O \left( \frac{n}{X \log X} \right),
\end{equation}
where for the last estimate we used Lemma~\ref{lem:psquared_sum_est}.

Overall, we have from \eqref{eq:sn_1} and \eqref{eq:sn_2} that
\[
    s(n)
    = c \cdot n + E_1(n) + E_2(n) + E_3(n).
\]
Using the bounds for the error terms from \eqref{eq:E1}, \eqref{eq:E2}, and \eqref{eq:E3}, 
and setting $X = X(n) = \log n$,
we get
\begin{align*}
    s(n)
    &= c \cdot n + O \left( \frac{n}{X \log X} \right) + O(X) + O(2^{\pi(X)})\\
    &= c \cdot n + O \left( \frac{n}{\log n \log \log n} \right) + O(\log n) + O(2^{\pi(\log n)}).
\end{align*}
The theorem follows upon noting that with Lemma~\ref{lem:primecounting_ub} for sufficiently large $n$ we get 
\[
    2^{\pi(\log n)} 
    \leq 2^{1.3 \log n / \log \log n}
    \leq n ^{1/\log \log n}
    \leq n^{1/2}
    = O \left( \frac{n}{\log n \log \log n} \right).
\]
\end{proof}

\begin{remark}
At first sight, the last estimate in the proof above seems very rough. However, increasing $X(n)$ from $\log n$ even to just $(\log n)^{1+\eps}$ does not work using the same arguments. 
In order to improve the bound, one would have to find stronger error bounds.
\end{remark}

\section{Differences between consecutive terms of $s(n)$ (proof of Theorem~\ref{thm:diffs})}\label{sec:diffs}

In this section we try to better understand how $s(n)$ changes as $n$ increases by 1. 

First, observe the following relation between elements in $S(n+1)$ and $S(n)$.

\begin{lemma}\label{lem:r_transfered}
For every $n \geq 1$ we have
\[
    S(n+1) \setminus \{0\} 
    \subseteq \{r+1 \colon r \in S(n)\}.
\]
\end{lemma}
\begin{proof}
Let $r \in S(n+1)\setminus\{0\}$. Then by Lemma~\ref{lem:member-equiv} the integer $n+1-r$ has a proper divisor $k \geq r+1$. 
Now since $n-(r-1) = n + 1 - r$ has a proper divisor $k \geq r+1 \geq r$, we have $r-1 \in S(n)$.     
\end{proof}

From this, we immediately see that the value of $s(n)$ can increase by at most one at a time:

\begin{proposition} \label{prop:diffs_ub}
For every $n \geq 1$ we have
\[
    s(n+1) \leq s(n) + 1.
\]    
\end{proposition}
\begin{proof}
We have
\[
    s(n+1) - 1
    = |S(n+1)| - 1
    \leq |S(n+1)\setminus \{0\}| 
    \leq |S(n)|
    = s(n),
\]
where we applied Lemma~\ref{lem:r_transfered} for the last inequality.
\end{proof}

Note that $0 = n \bmod 1$ is always in $S(n+1)$ and recall Lemma~\ref{lem:r_transfered} and it's proof. We can interpret it in the following way: The set $S(n+1)$ consists precisely of the element $0$ and all the elements $r+1$ that were ``transferred'' from $S(n)$ to $S(n+1)$ by being increased by $1$.
Thus, in order to understand the difference $s(n+1) - s(n)$ we need to understand how many elements $r$ were not transferred. 
We denote the set of ``not transferred elements'' by
\[
    T(n,n+1) 
    := \{ r \colon r \in S(n) \text{ and } r+1 \notin S(n+1)\}.
\]
Then we have
\begin{equation}\label{eq:T}
    s(n+1) = s(n) + 1 - \abs{T(n,n+1)}.
\end{equation}

Next, we characterize the set $T(n,n+1)$.

\begin{lemma}\label{lem:r_not_transf}
We have $r \in T(n,n+1)$ if and only if 
$r+1$ is the largest proper divisor of $n-r$.
\end{lemma}
\begin{proof}
This follows directly from Lemma~\ref{lem:member-equiv}: $r \in S(n)$ is equivalent to $n - r$ having a proper divisor $\geq r+1$,
and $r+1 \notin S(n+1)$ is equivalent to $(n+1) - (r+1) = n - r$ not having a proper divisor $\geq r+2$.
\end{proof}

In order to count such occurrences, we will use the next simple lemma.

\begin{lemma}\label{lem:largest_proper_div}
Let $n, d$ be positive integers.
Then $d$ is the largest proper divisor of $n$ if and only if $n/d = p$ for some prime $p$ and every prime factor of $d$ is $\geq p$.
\end{lemma}

Now for even $n$, the situation is rather simple:

\begin{proposition}\label{prop:incr_n_even}
Let $n\geq 2$ be even. Then
\[
    T(n,n+1) = \begin{cases}
        \{ (n-2)/3 \}, &\text{if } n \equiv 2 \pmod{3};\\
        \emptyset, &\text{otherwise}.
        \end{cases}
\]
In particular,
\[
    s(n+1) = \begin{cases}
        s(n), &\text{if } n \equiv 2 \pmod{3};\\
        s(n) + 1, &\text{otherwise}.
    \end{cases}
\]
\end{proposition}
\begin{proof}
By Lemma~\ref{lem:r_not_transf} we have $r \in T(n,n+1)$ if and only if $r+1$ is the largest proper divisor of $n-r$.
By Lemma~\ref{lem:largest_proper_div} this is the case if and only if $(n-r)/(r+1) = p$ for some prime $p$ and every prime factor of $r+1$ is $\geq p$.
Let us find all $r$ where this is the case.
Assume first that $r$ is odd. Then $(n-r)/(r+1) = \odd / \even = p$, which is impossible.
Now assume that $r$ is even, and we have $(n-r)/(r+1) = \even / \odd = p$.
This means that $p = 2$ and $n = 3r + 2$. So this happens if and only if $n \equiv 2 \pmod{3}$ and $r = (n-2)/3$.

The second part of the lemma now follows immediately from formula~\eqref{eq:T}.
\end{proof}

For odd $n$, we will prove that $\abs{T(n,n+1)}$ is at most double logarithmic in $n$.

Let us first use Lemmas~\ref{lem:r_not_transf} and \ref{lem:largest_proper_div} to characterize the sets $T(n,n+1)$ more precisely.

\begin{lemma}\label{lem:r_not_transf_primes}
Assume that $n + 1 = p_1^{x_1} \cdots p_\ell^{x_\ell}$ is the prime factorization of $n+1$ with $p_1 < \dots < p_\ell$. Let $1 \leq r < (n-2)/3$.
Then $r \in T(n,n+1)$ if and only if all of the following conditions are satisfied:
\begin{enumerate}[(i)]
    \item $r +1 = p_I^{x_I} \cdots p_\ell^{x_\ell}$ for some index $2 \leq I \leq \ell$;
    \item $(n+1)/(r+1) - 1 = p_1^{x_1} \cdots p_{I-1}^{x_{I-1}} - 1 = p$ for some prime $p$;
    \item $p \leq p_{I}$.
\end{enumerate}
Moreover, $r = 0 \in T(n,n+1)$ if and only if $n = p_1^{x_1} \cdots p_\ell^{x_\ell} - 1$ is prime.
Finally, if $n \equiv 2 \pmod{3}$, then $r = (n-2)/3 \in T(n,n+1)$.
\end{lemma}
\begin{proof}
By Lemma~\ref{lem:r_not_transf} we have $r \in T(n,n+1)$ if and only if $r + 1$ is the largest proper divisor of $n-r$. 
By Lemma~\ref{lem:largest_proper_div} this is equivalent to
\begin{align}
    &\frac{n-r}{r+1}
    = \frac{n+1}{r+1} - 1 
    = p
    \label{eq:nrp_part1}\\
    & \text{and every prime factor of $r+1$ is at least of size $p$}. \label{eq:nrp_part2}
\end{align}
Let $1 \leq r < (n-2)/3$ and assume that $n + 1 = p_1^{x_1} \cdots p_\ell^{x_\ell}$. We check that \eqref{eq:nrp_part1} and \eqref{eq:nrp_part2} are equivalent to the three conditions in the statement of the lemma.

Assume first that \eqref{eq:nrp_part1} and \eqref{eq:nrp_part2} hold. Then since $r+1 \mid n+1$, we have 
\[
    r+1 = p_1^{y_1} \cdots p_\ell^{y_\ell}, 
    \quad \text{with } 0 \leq y_i \leq x_i \text{ for } 1 \leq i \leq \ell.
\]
Let $I$ be the index such that $p_1 < \dots < p_{I-1} < p \leq p_{I} < \dots < p_\ell$ (where $I = 1$ is allowed and means that $p \leq p_1$).
Then \eqref{eq:nrp_part2} implies that $y_1 = \dots = y_{I-1} = 0$.
Moreover, for every $I \leq i \leq \ell$ we can write
\[
    \frac{n+1}{r+1} 
    = p_i^{x_i - y_i} \cdot A
    = p+1,
\]
with some integer $A\geq 1$. 
Then $p_{I} \geq p$ implies $x_i - y_i = 0$ for all $i = I, \ldots, \ell$, except 
if $p_i = 3$, $x_i - y_i = 1$, $A = 1$ and $p = 2$. The exceptional situation is $(n+1)/(r+1) = 3$, which is equivalent to $r = (n-2)/3$ and was excluded.
In the other situations, we have $y_{I} = x_{I}, \ldots, y_{\ell} = x_{\ell}$,
and the three conditions in the lemma are clearly satisfied (the condition $I \geq 2$ follows from $(n+1)/(r+1) = p+1 > 1$).

Conversely, it is clear that the three conditions in the lemma imply \eqref{eq:nrp_part1} and \eqref{eq:nrp_part2}.

The statement for $r=0$ holds because always $0 \in S(n)$ and by Lemma~\ref{lem:member-equiv} we have $1 \in S(n+1)$ if and only if $n = (n+1)-1$ is composite.

Finally, assume $r = (n-2)/3$ is an integer. Then $n \bmod (n+1)/3 = r$ and so $r \in S(n)$. On the other hand, $r+1 > (n-2)/3 = \floor{(n+1-2)/3}$ and so by Lemma~\ref{lem:upper_bound_nover3} we have $r+1 \notin S(n+1)$.
\end{proof}

\begin{proposition} \label{prop:diff_lb}
For every odd $n \geq 1$ we have
\[
    s(n) - s(n+1) 
    = O(\log \log n).
\]
\end{proposition}
\begin{proof}
In view of formula~\eqref{eq:T} it suffices to prove that
\[
    \abs{T(n,n+1) \setminus \{0, (n-2)/3\} } 
    = O(\log \log n).
\]
In other words, we want to show that there are at most $O(\log \log n)$ numbers $r$ satisfying the three conditions in Lemma~\ref{lem:r_not_transf_primes}. 
Fix an odd $n\geq 1$ an let 
\[
    n+1 = p_1^{x_1} \cdots p_\ell^{x_\ell}
\]
be the prime factorization of $n$ with $p_1 < \dots < p_\ell$.
Then by Lemma~\ref{lem:r_not_transf_primes}, $r+1$ must be of the shape $r +1 = p_I^{x_I} \cdots p_\ell^{x_\ell}$ for some index $2 \leq I \leq \ell$.
Moreover, for these indices $I$, we have 
\[
    p_I
    \geq p
    = p_1^{x_1} \cdots p_{I-1}^{x_{I-1}} - 1
    \geq p_2 \cdots p_{I-1}.
\]
Assume that there are $t$ such numbers $r$, corresponding to the indices $I_1< \dots < I_t$.
Then we can weaken the above inequality to
\[
    p_{I_j}
    \geq p_{I_2} p_{I_3} \cdots p_{I_{j-1}}.
\]
Using this inequality inductively, we get
\begin{align*}
    n
    &= p_{1}^{x_{1}} \cdots p_\ell^{x_\ell} - 1
    \geq p_{I_2} p_{I_3} \cdots p_{I_t}\\
    &\geq (p_{I_2} p_{I_3} \cdots p_{I_{t-1}})^2 
    \geq (p_{I_2} p_{I_3} \cdots p_{I_{t-2}})^4
    \geq \dots 
    \geq p_{I_2}^{2^{t - 2}}
    \geq 3^{2^{t - 2}}.
\end{align*}
This implies
\[
    t - 2
    \leq \log_3 \log_2 (n),
\]
and so $t = O(\log \log n)$, as desired.
\end{proof}

\begin{proposition}\label{prop:diffs_unbounded}
We have
\[
    \liminf_{n\to \infty} s(n+1) - s(n) = -\infty.
\]
\end{proposition}
\begin{proof}
Let us set
\[
    t(n) := \abs{T(n)}. 
\]
In view of formula~\eqref{eq:T}, our strategy is to use Lemma~\ref{lem:r_not_transf_primes} to recursively construct numbers $n^{(1)}, n^{(2)}, \ldots$ with the properties
\begin{itemize}
    \item $n^{(j)}$ is prime and
    \item $t(n^{(j)}) \geq j$.
\end{itemize}
For good intuition, let us set the first three values straight away: 
\[
    n^{(1)} := 3 = 2^2-1, \quad
    n^{(2)} := 11 = 2^2 \cdot 3 - 1, \quad 
    n^{(3)} := 131 = 2^2 \cdot 3 \cdot 11 - 1.
\]
This works because of Lemma~\ref{lem:r_not_transf_primes} and the following facts:
$2^2 - 1$ is prime and $2^2 - 1 \geq 3$, and $2^2 \cdot 3 - 1$ is prime and $2^2 \cdot 3 - 1\geq 11$, and also $2^2 \cdot 3 \cdot 11 - 1$ is prime.

Now assume we have already constructed $n^{(j)}$ for some $j \geq 3$.
Set
\[
    P := \prod_{p \leq n^{(j)}} p.
\]
For reasons that will become apparent in \eqref{eq:arith_prog},
we want to choose an integer $y$ with the following properties: 
\begin{enumerate}[(a)]
    \item \label{it:gcd1} $\gcd(xP + y,\, P) = 1$ for all integers $x$. 
    This is equivalent to $y \not\equiv 0 \pmod{p}$ for all $p \mid P$.
    \item \label{it:gcd2} $\gcd(y(n^{(j)}+1)-1,\, P) = 1$. This is equivalent to $y(n^{(j)}+1) \not \equiv 1 \pmod{p}$ for all $p\mid P$ with $p \nmid n^{(j)}+1$. 
\end{enumerate}
Overall, we want $y$ to satisfy
\[
    y \not\equiv 0, (n^{(j)}+1)^{-1} \pmod{p} 
    \qquad \text{for all } p \mid P.
\]
For each $p\geq 3$ this is clearly possible, since at most two residue classes need to be avoided.
For the case $p=2$, note that since $n^{(j)}$ is a prime, $n^{(j)} + 1$ is even, so $p \mid n^{(j)} + 1$ and we only need to avoid $y \equiv 0 \pmod{2}$.

Therefore, by the Chinese remainder theorem, there exists an integer $1\leq y \leq P$ with properties (\ref{it:gcd1}) and (\ref{it:gcd2}).

We fix such an integer $y$ and consider the arithmetic progression
\begin{equation}\label{eq:arith_prog}
    a_x
    := (n^{(j)} + 1) (xP + y) - 1 
    = x P (n^{(j)}+1) + y(n^{(j)}+1) - 1,
    \quad x\geq 1.
\end{equation}
Now our property (\ref{it:gcd2}) implies that
\[
    \gcd( P(n^{(j)}+1),\ y(n^{(j)}+1) - 1) = 1.
\]
Thus, by Dirichlet's prime number theorem, there exists an integer $x \geq 1$ such that $a_x$ is a prime.
We set 
\[
    n^{(j+1)} 
    := (n^{(j)} + 1) (xP + y) - 1
    = a_x.
\]
By construction, $n^{(j+1)}$ is prime, and we finally only need to check that indeed $t(n^{(j+1)}) \geq t(n^{(j)})+1 \geq j+1$.

It is easy to see that by construction $n^{(j+1)} \equiv 2 \pmod{3}$ for all $j\geq 2$ and $n^{(j+1)}$ is prime, so the special cases from Lemma~\ref{lem:r_not_transf_primes}, namely $r = 0$ and $r = (n^{(j+1)}-2)/3$, are always in $T(n^{(j+1)}, n^{(j+1)}+1)$.

Moreover, let us write $n^{(j)} = p_1^{x_1} \cdots p_\ell^{x_\ell} - 1$ with $p_1 < \dots < p_\ell$. Then  we have $n^{(j+1)} = p_1^{x_1} \cdots p_\ell^{x_\ell} \cdot Q - 1$, where all prime factors in $Q = xP + y$ are larger than $p_\ell$ by property~(\ref{it:gcd1}) and the definition of $P$. 
Thus, if $r$ satisfied the three conditions in Lemma~\ref{lem:r_not_transf_primes} for $n^{(j)}$, then $r' = (r-1)\cdot Q + 1$ satisfies the three conditions in Lemma~\ref{lem:r_not_transf_primes} for $n^{(j+1)}$.
The increase comes from the fact that now $r' = Q - 1$ satisfies the three conditions as well.
\end{proof}

\begin{proof}[Proof of Theorem~\ref{thm:diffs}]
Combine Propositions~\ref{prop:diffs_ub}, \ref{prop:incr_n_even}, \ref{prop:diff_lb}, and \ref{prop:diffs_unbounded}.
\end{proof}

\section{Iterated remainder sets and their relation to the ``n mod a problem''}\label{sec:Pierce}

Recall that we defined the sets of iterated remainders of $n$ inductively by
\begin{align*}
    S_0(n) &: = \{1,2, \ldots, \floor{n/2}\}
    \quad \text{and} \\
    S_{j}(n) &:= \{n \bmod{k} \colon k \in S_{j-1}(n)\setminus \{0\}\}
    \quad \text{for } j \geq 1.
\end{align*}

These sets are related to an older open problem about the length of Pierce expansion.
This problem was first studied by Shallit \cite{Shallit1986}, and can be phrased in the following way.

Fix a positive integer $n$ and choose another  integer $1 \leq a \leq n$.
Set $a_0 := a$ and $a_{j+1} := n \bmod a_j$ for $j\geq 0$ and as long as $a_j > 0$. 
For example, for $(n,a) = (35,22)$, we get $a_0 = 22, a_1 = 13, a_2 = 9, a_3 = 8, a_4 = 3, a_5 = 2, a_6 = 1, a_7 = 0$. 
Now let us define $P(n,a)$ to be the integer $t$ such that $a_t = 0$. So, for example, $P(35,22) = 7$. Finally, let us set
\[
    P(n) := \max_{1 \leq a \leq n} P(n,a).
\]
The problem is to obtain upper and lower bounds for $P(n)$ in terms of $n$.
In particular, experiments suggest that the upper bound should be sublinear; but this seems to be hard to prove.
The best known bounds are due to Chase and Pandey \cite{ChasePandey2022}, who slightly improved the bounds by Erd\H{o}s and Shallit \cite{ErdosShallit1991}:
We have
\begin{equation}\label{eq:P-bounds}
    \frac{\log n}{\log \log n}
    \ll P(n) 
    \ll n^{\frac{1}{3} - \frac{2}{177} + \eps}
\end{equation}
for sufficiently large $n$.

Now this problem directly relates to our sets $S_j(n)$ via the next simple lemmas.

\begin{lemma}\label{lem:S-chains}
Let $j \geq 1$ and $n \geq 1$. Then $r \in S_j(n)$ if and only if there exists an integer $\floor{n/2}+1 \leq a \leq n$ such that in the above notation $a_{j+1} = r$.
\end{lemma}
\begin{proof}
This follows directly from the definitions. The index shift comes from the fact that we are assuming $\floor{n/2}+1 \leq a \leq n$ and so $a_1$ can be exactly every element from $S_0(n)$.
\end{proof}

\begin{lemma}
The following statements are equivalent:
\begin{enumerate}
    \item $P(n) = t$;
    \item $S_{t+1}(n) = \{0\}$;
    \item $\abs{S_{t+1}(n)} = 1$ and $\abs{S_{t+1+j}(n)} = 0$ for all $j \geq 1$.
\end{enumerate}
\end{lemma}
\begin{proof}
When computing $P(n) = \max_{1 \leq a \leq n} P(n,a)$, we may restrict ourselves to $\floor{n/2}+1 \leq a \leq n$, since if $a< \floor{{n/2}+1}$, the starting value $n-a$ gives $P(n,n-a) = P(n,a)+1$.

Now the equivalences follow from Lemma~\ref{lem:S-chains} and the fact that $S_j(n) = \emptyset$ if and only if $S_{j-1} = \emptyset$ or $S_{j-1} = \{0\}$.
\end{proof}

\begin{remark}
The known bounds \eqref{eq:P-bounds} imply that for sufficiently large $n$, we have $\abs{S_j(n)} = 0$ for all $j \geq n^{1/3}$.
On the other hand, there exists a constant $c$ such that for all sufficiently large $n$ we have $\abs{S_j(n)} \geq 1$ for $j \leq c \log n / \log \log n$.
\end{remark}

\section{Bounds for iterated remainders}\label{sec:iterated_results}

Recall that we defined
\[
    s_j(n) := \abs{S_j(n)}.
\]
In this section we prove upper and lower bounds for $s_j(n)$.

We start with a simple upper bound.

\begin{lemma}\label{lem:iter_ub}
For all $j \geq 0$ and $n \geq 1$ we have
\[
    s_j(n) -1 
    \leq \max S_j(n)
    \leq \frac{n}{j+2}.
\]
\end{lemma}
\begin{proof}
The first inequality is clear since $\min S_j(n) = 0$.
We show the second inequality by induction.
For $j=0$ this is clearly satisfied by definition.
Now assume that $\max S_j(n) \leq n/(j+2)$ for some $j\geq 0$.
Then for every $r \in S_{j+1}(n)$ there exists a $k \in S_j(n)$ with $k>r$ such that $n = qk + r$.
By induction hypothesis we have $k\leq n/(j+2)$; hence $q \geq j+2$.
Now
\[
    r+1 
    \leq k 
    = \frac{n-r}{q}
\]
implies
\[
    r 
    \leq \frac{n-q}{q+1}
    \leq \frac{n}{j+3},
\]
as desired.
\end{proof}

\begin{remark}
The arguments from \cite{ErdosShallit1991} for the upper bound on $P(n)$ in fact give stronger upper bounds when $j$ is relatively large compared to $n$. For example, since $\max S_j(n)$ strictly decreases as $j$ increases, one can show that $\max S_j(n)$ is roughly bounded by $2\sqrt{n} - j$ for $j > \sqrt{n}$.
One can do even better (see \cite[proof of Theorem 2]{ErdosShallit1991}), using the fact that the number of divisors of $m$ is $O(m^\eps)$. It turns out that for $j>n^{1/3 + \eps}$ we get roughly the upper bound $2 n^{2/3 + \eps} -  j \cdot n^{1/3}$.
\end{remark}

Finally, we prove a lower bound. In particular, we show that the sequence $s_j(n)$ grows linearly (even if $\lim_{n \to \infty} s_j(n)/n$ might not exist).

\begin{lemma}\label{lem:iter_lb}
For every $j\geq 0$ and $n \geq N(j)$ there exists an integer $x_j = x_j(n)$ such that
\begin{equation}\label{eq:Sj_subset}
    S_j(n)
    \supseteq \{ r \colon j \leq r \leq \frac{n - j-1}{j+2},\, r \equiv x_j  \ (\bmod{\ (j+1)!}) \}.
\end{equation}
In particular,
\begin{equation}\label{eq:sj_liminf}
    \liminf_{n \to \infty} \frac{s_j(n)}{n}
    \geq \frac{1}{(j+2)!}.
\end{equation}
    
\end{lemma}
\begin{proof}
Fix some $J\geq 0$ and let $n \geq N(J)$. We want to prove \eqref{eq:Sj_subset} for $j = 0,1, \ldots, J$ with finite induction.
For $j=0$ the inclusion \eqref{eq:Sj_subset} is clearly true with $x_0 = 0$, since $S_0(n) = \{0,1, \ldots, \floor{n/2} \}$.
Now assume that \eqref{eq:Sj_subset} holds for some $0 \leq j \leq J-1$. 
Our goal is to show
\[
    S_{j+1}(n)
    \supseteq \{ r \colon j+1 \leq r \leq \frac{n - j - 2}{j+3},\, r \equiv n -(j+2)x_j \pmod{(j+2)!} \};
\]
i.e., we set $x_{j+1} = n - (j+2) x_j$.
Assume $r$ is in the set on the right hand side. 
Then the condition on the residue class implies that we can write
$r = n - (j+2)x_j - q(j+2)!$ with some integer $q$. This implies 
\[
    n-r 
    = (j+2)(x_j + q(j+1)!)
    = (j+2)k,
\]
where we set $k := x_j + q(j+1)!$.
Thus, in order to prove $r \in S_{j+1}(n)$ it suffices to show that $k \in S_j(n)$ and $k \geq r+1$.

Since $n-r = (j+2)k$, the condition $k \geq r+1$ is equivalent to $n-r \geq (j+2)(r+1)$. This is equivalent to $r\leq (n-j-2)/(j+3)$, which is satisfied by assumption.

Clearly, $k \equiv x_j \bmod{(j+1)!}$, so in order to show $k \in S_j(n)$ we only need to check $j \leq k \leq (n-j-1)/(j+2)$.
First, $k \leq (n-j-1)/(j+2)$ is equivalent to $n-r \leq (j+2)(n- j-1)/(j+2) = n - j-1$ so this indeed holds for $r\geq j+1$.
Finally, $k \geq j$ is equivalent to $n-r\geq (j+2)j$, which holds since $r \leq (n-j-2)/(j+3) \leq n - (j+2)j$ for sufficiently large $n$.

The bound \eqref{eq:sj_liminf} follows immediately from \eqref{eq:Sj_subset}, since for large $n$ the interval length is $(n-j-1)/(j+2) - j + 1 \sim n/(j+2)$ and every $(j+1)!$-th number is included.
\end{proof}

\begin{proof}[Proof of Theorem~\ref{thm:iterated_bounds}]
Combine Lemmas~\ref{lem:iter_ub} and \ref{lem:iter_lb}.
\end{proof}

\section{Numerical experiments and open problems}\label{sec:problems}

Recall that in Theorem~\ref{thm:asymp} we proved 
\[
    s(n) = c \cdot n + O \left( \frac{n}{\log n \log \log n} \right).
\]
The bound for the error term seems very large.
We have computed the values for $s(n)$ for $n\leq 10^7$ and determined the points $(n, s(n)-c \cdot n)$ where $s(n) - c \cdot n$ reaches a new maximum or minimum. These points are plotted in Figure~\ref{fig:s_error}, together with the graph of $2 n^{1/3}$. This suggests that perhaps the correct bound for the error term might be $O(n^{1/3})$. 
In any case, we propose the following problem.

\begin{figure}[h]
  \centering
  \includegraphics[width=0.7\textwidth]{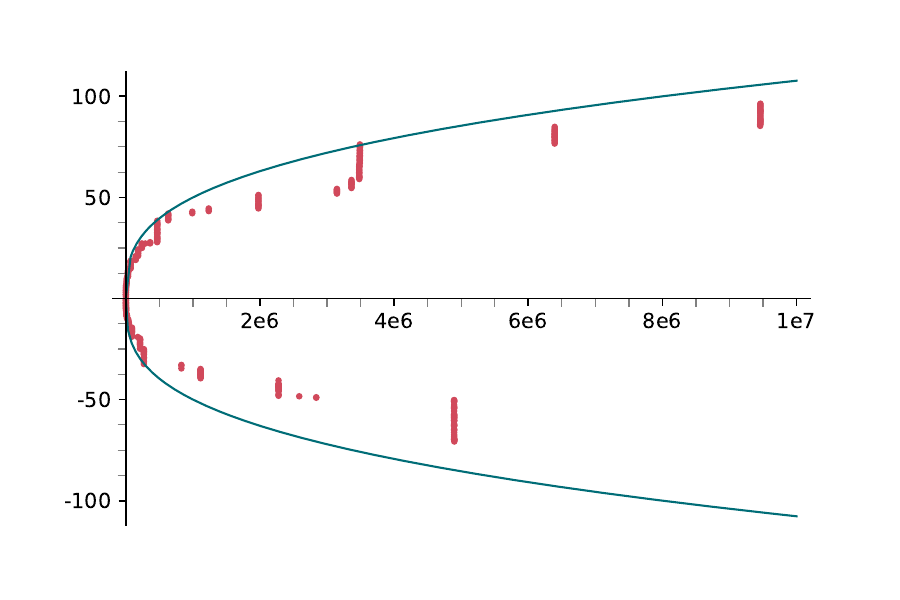}
  \caption{The function $2n^{1/3}$ and the largest deviations of $s(n)$ from $c\cdot n$.}
  \label{fig:s_error}
\end{figure}

\begin{problem}
Improve the bound for the error term for $s(n)$ in Theorem~\ref{thm:asymp}.
\end{problem}

For the iterated remainder sets, recall that Theorem~\ref{thm:iterated_bounds} says that
\[
    \frac{1}{(j+2)!}
    \leq \liminf_{n \to \infty} \frac{s_j(n)}{n}
    \leq \limsup_{n \to \infty} \frac{s_j(n)}{n}
    \leq \frac{1}{j+2}.
\]
In particular, there is a gap between our lower and our upper bound.
Indeed, numerical experiments strongly suggest that $\lim_{n\to \infty} s_j(n)/n$ does not exist for $j\geq 2$.
The three plots in Figure~\ref{fig:s123} show the values of $s_1(n), s_2(n), s_3(n)$, respectively, for $n \leq 10^4$.
For $s_2(n)$ and $s_3(n)$ we see some ``bands'' of values. The blue points correspond to $n \equiv 0 \bmod 6$, the green points correspond to $n \equiv 2,4 \bmod 6$, the yellow points to $n \equiv 3 \bmod 6$, and the red points to $n \equiv 1,5 \bmod 6$.
It seems that really the precise divisibility properties of $n$ determine the value of $s_2(n), s_3(n)$. To support this further, in Figure~\ref{fig:1mod6} we have plotted the values of $s_2(n)$ only for $n \equiv 1 \bmod 6$ in the range $[6\cdot 10^4, 6 \cdot 10^4 + 10^3]$. In particular, we only consider numbers $n$ not divisible by $2$ and $3$. Indeed, the numbers $n$ that are divisible by $5$ (points coloured red) yield the smallest relative values.
In any case, the easiest problem in this context might be Problem~\ref{probl:lim}.

\begin{figure}[h]
  \centering

  \begin{subfigure}[b]{0.3\textwidth}
    \centering
    \includegraphics[width=\textwidth]{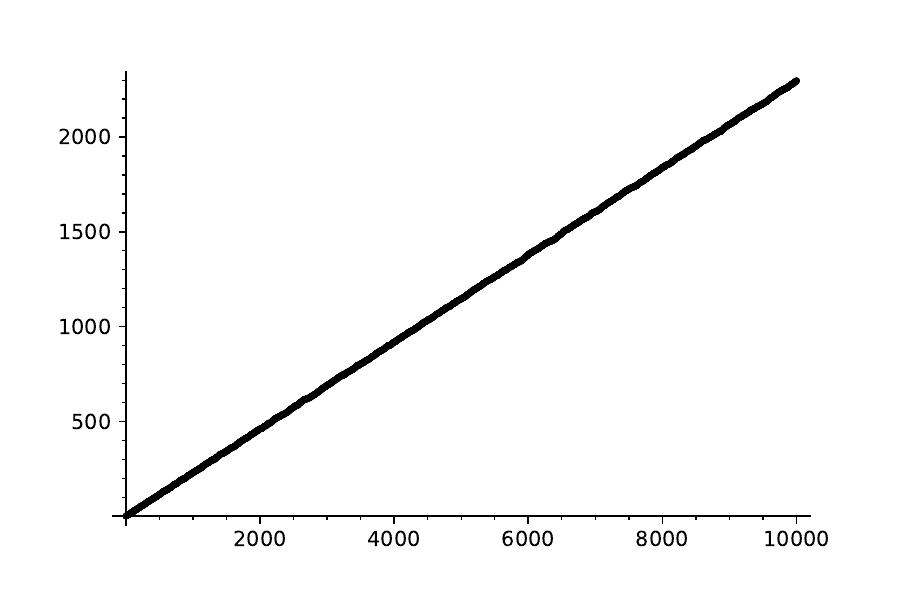}
    \caption{$s_1(n)$}
    \label{fig:s1}
  \end{subfigure}
  \hfill
  \begin{subfigure}[b]{0.3\textwidth}
    \centering
    \includegraphics[width=\textwidth]{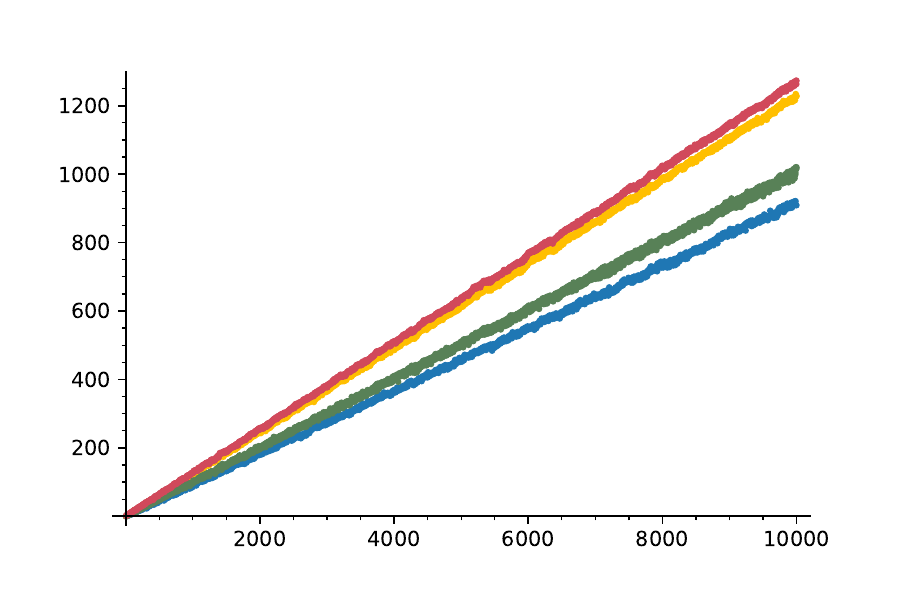}
    \caption{$s_2(n)$}
    \label{fig:s2}
  \end{subfigure}
  \hfill
  \begin{subfigure}[b]{0.3\textwidth}
    \centering
    \includegraphics[width=\textwidth]{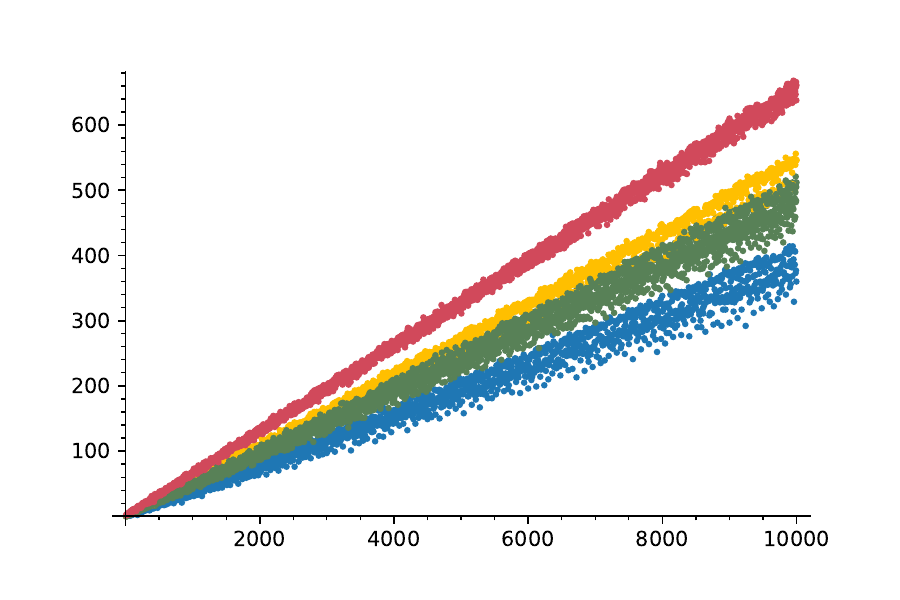}
    \caption{$s_3(n)$}
    \label{fig:s3}
  \end{subfigure}

  \caption{Plots of $s_j(n)$ for $n=1,2,3$; colours according to divisibility by $2$ and $3$.}
  \label{fig:s123}
\end{figure}

\begin{figure}[h]
  \centering
  \includegraphics[width=0.3\textwidth]{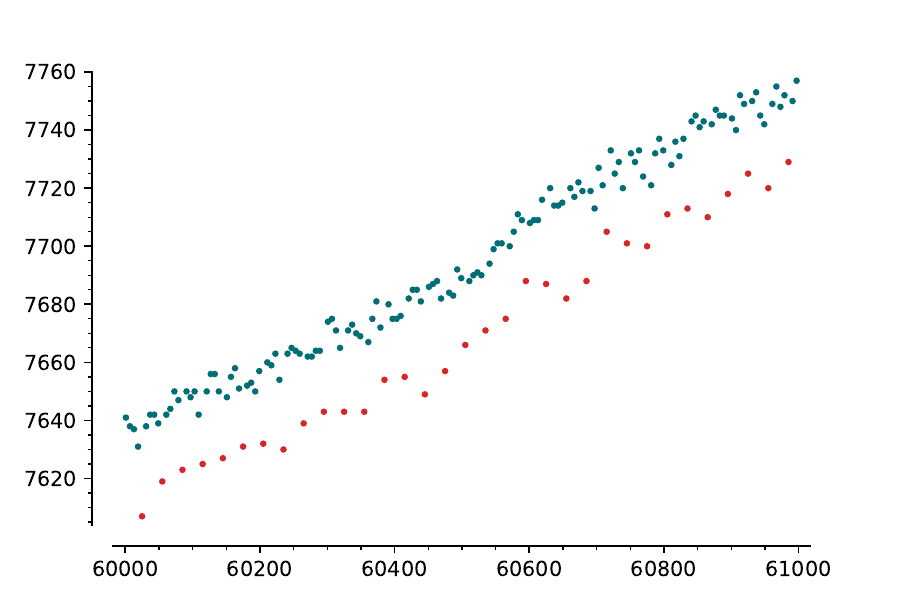}
  \caption{$s_2(n)$ for $n \equiv 1 \bmod 6$; the points where $n$ is divisible by $5$ are colored red.}
  \label{fig:1mod6}
\end{figure}

\begin{problem}\label{probl:lim}
Prove that for $j\geq 2$ the limit $\lim_{n\to \infty} s_j(n)/n$ does not exist.
\end{problem}

\section*{Acknowledgements}
We thank Jeffrey Shallit for his encouragement, for sharing his numbers up to $s(10^7)$ with us, and for suggesting to consider iterated sets.

\bibliographystyle{habbrv}
\bibliography{refs}

\end{document}